\documentclass[10pt]{amsart}
\usepackage{amsmath,amssymb,verbatim}
\usepackage[latin1]{inputenc}
\usepackage{enumerate}
\usepackage{pdflscape}
\usepackage{amsthm}
\usepackage{mathrsfs}
\usepackage{color}
\usepackage[normalem]{ulem}
\usepackage{tikz}
\usetikzlibrary{matrix}
\usepackage[all]{xy}
\usepackage{mathtools}

\newtheorem{theorem}{Theorem}[section]
\newtheorem{lemma}[theorem]{Lemma}
\newtheorem{proposition}[theorem]{Proposition}
\newtheorem{example}[theorem]{Example}

\newtheorem{corollary}[theorem]{Corollary}
\newtheorem{remark}[theorem]{Remark}

\newcommand{\Z}{{\mathbb Z}}
\newcommand{\Q}{{\mathbb Q}}

\newcommand{\C}{{\mathbb C}}

\newcommand{\PSL}{{\rm PSL}}

\newcommand{\ind}{{\rm ind}}

\newcommand{\matriz}[1]{\begin{array} #1 \end{array}}

\newcommand{\GEN}[1]{\langle #1 \rangle}

\newcommand{\V}{\mathrm{V}}

\newenvironment{proofof}{\par\noindent \textit{Proof of }}{\qed\par\bigskip}

\newcommand{\pa}[2]{\varepsilon_{#2}(#1)}

\DeclareMathOperator{\Cl}{Cl}

\DeclareMathOperator{\Irr}{Irr}
\DeclareMathOperator{\IBr}{IBr}
\DeclareMathOperator{\IRR}{IRR}
\DeclareMathOperator{\VPA}{VPA}

\DeclareMathOperator{\PAP}{PAP}

\begin{document}
\title[Partial Augmentations Power property]{Partial Augmentations Power property: \\
A Zassenhaus Conjecture related problem}

\author{Leo Margolis and Ángel del Río}
\thanks{This research is partially supported by the European Commission under H2020 Grant 705112-ZC, by the Spanish Government under Grant MTM2016-77445-P with "Fondos FEDER" and, by Fundación Séneca of Murcia under Grant 19880/GERM/15.}



\begin{abstract}
Zassenhaus conjectured that any unit of finite order in the integral group ring $\mathbb{Z}G$ of a finite group $G$ is conjugate in the rational group algebra of $G$ to an element in $\pm G$. We review the known weaker versions of this conjecture and introduce a new condition, on the partial augmentations of the powers of a unit of finite order in $\mathbb{Z}G$, which is weaker than the Zassenhaus Conjecture but stronger than its other weaker versions.

We prove that this condition is satisfied for units mapping to the identity modulo a nilpotent normal subgroup of $G$. Moreover, we show that if the condition holds then the HeLP Method adopts a more friendly form and use this to prove the Zassenhaus Conjecture for a special class of groups.
\end{abstract}

\maketitle

\section{Open problems on torsion units of integral group rings}

The structure of the group of units $\mathrm{U}(\mathbb{Z}G)$ of an integral group ring $\mathbb{Z}G$ of a group $G$ has given rise to many interesting results and even more questions.
In particular for finite $G$, the units of finite order in $\mathrm{U}(\mathbb{Z}G)$ were already studied by G.~Higman, who showed that for abelian $G$ all are of the form $\pm g$ for some $g \in G$ \cite{Higman1940}. Note that, denoting by $\mathrm{V}(\mathbb{Z}G)$ the elements of $\mathrm{U}(\mathbb{Z}G)$ whose coefficients sum up to $1$, we have $\mathrm{U}(\mathbb{Z}G) = \pm \mathrm{V}(\mathbb{Z}G)$. So to study the unit group of $\mathbb{Z}G$ it is sufficient to study $\mathrm{V}(\mathbb{Z}G)$.

For non-abelian $G$ already the smallest examples show that a result in the spirit of Higman is not valid and the strongest result one could hope for was put forward by H.~J.~Zassenhaus as a conjecture.

\begin{quote}
\textbf{Zassenhaus Conjecture (ZC):} \cite{Zassenhaus} Let $G$ be a finite group and $u$ a unit
in $\mathrm{V}(\mathbb{Z}G)$ of finite order. Then there exists a unit $x$ in the rational group algebra of $G$ and an element $g \in G$ such that $x^{-1}ux =
g$.
\end{quote}

If for $u \in \mathrm{V}(\mathbb{Z}G)$ such $x$ and $g$ exist we say that $u$ and $g$ are rationally conjugate.
The Zassenhaus Conjecture has been proven for some classes of solvable groups, such as nilpotent groups \cite{Weiss1991}, cyclic-by-abelian groups \cite{CaicedoMargolisdelRio2013}, or groups which possess a normal Sylow subgroup with abelian complement \cite{Hertweck2006}. It has also been proven for some particular non-solvable groups \cite{LutharPassi1989, HertweckA6, 4primaryII, FermatMersenne}.

Recently a metabelian counterexample to the Zassenhaus Conjecture has been obtained by Eisele and Margolis \cite{EiseleMargolis}.
It still seems almost hopeless to classify the groups for which the Zassenhaus Conjecture holds in wide classes of groups as e.g. solvable groups, or even metabelian groups.
Therefore also weaker versions of this conjecture are studied.
For example Kimmerle proposed to study the following problem \cite{Ari} which has been studied in \cite{KimmerleMargolis17}.

\begin{quote}
\textbf{Kimmerle Problem, (KP)}
Given a finite group $G$ and an element $u$ of finite order in $\V(\Z G)$, is there a finite group $H$ containing $G$ as subgroup such that $u$ is conjugate in $\Q H$ to an element of $G$?
\end{quote}

Recall that the \emph{spectrum} of a group $G$ is the set of orders of its torsion elements.
The \textit{prime graph} of $G$ has vertices labelled by the primes appearing as orders of torsion elements of $G$ and two vertices $p$ and $q$ are connected if and only if $G$ possesses an element of order $pq$.
Clearly, if the Zassenhaus Conjecture has a positive solution for a group $G$ then $G$ and $\V(\Z G)$ have the same spectrum and hence the same prime graph, too. This
yields the following problems.
\begin{quote}
\textbf{The Spectrum Problem (SP):} (see e.g. \cite[Problem 8]{Sehgal1993}) Given a finite group $G$, do $G$ and $\V(\Z G)$ have the same spectrum?
\end{quote}
\begin{quote}
\textbf{The Prime Graph Question (PQ):} (introduced in \cite{Kimmerle2006}) Given a finite group $G$, do $G$ and $\V(\Z G)$ have the same prime graph?
\end{quote}

The Spectrum Problem and the Prime Graph Question have been answered for much wider classes than the Zassenhaus Conjecture. The Spectrum Problem has a positive
answer for solvable groups \cite{HertweckOrders} and Frobenius groups \cite[Corollary 2.5]{KimmerleKonovalov2016}. The Prime Graph Question also holds for many sporadic simple groups (see e.g. \cite{BovdiKonovalovConway}) or
some infinite series of almost simple groups \cite{4primaryI}. Moreover in contrast to the Zassenhaus Conjecture and the Spectrum Problem, for the Prime Graph
Question a reduction theorem, to almost simple groups, is known \cite[Theorem~2.1]{KimmerleKonovalov2016}. The role of the prime graph for the structure of
integral group rings is described in detail in \cite{Kimmerle2006}. For an overview of the known results see also \cite{BaechleKimmerleMargolisDFG}.
Without putting assumptions on $G$, the most far reaching result regarding the Spectrum Problem is a theorem by Cohn-Livingstone which states that the exponents of $G$ and $\V(\Z G)$ coincide \cite{CohnLivingstone}.

In fact Cohn and Livingstone showed that certain sums of coefficients of torsion elements in $\V(\Z G)$ satisfy some congruences \cite[Proposition~4.1]{CohnLivingstone}.
Before we state these congruences denote by $\pa{a}{X}$ the sum of the coefficients of $a$ at $X$, i.e.
    $$\pa{a}{X} = \sum_{x \in X} a_x$$
for $a=\sum_{x\in G} a_x x\in \Z G$, with $a_x\in \Z$ for each $x$, and $X$ a subset of $G$.
If $X$ is a conjugacy class in $G$ then $\pa{a}{X}$ is called the
\textit{partial augmentation} of $a$ at $X$. The importance of partial augmentations in the study of the Zassenhaus Conjecture and related questions was first realized by Marciniak, Ritter, Sehgal and Weiss who proved the following: A torsion element $u \in \V(\Z G)$ is rationally conjugate to an element of $G$ if and
only if $\pa{v}{C} \geq 0$ for every $v\in \GEN{u}$ and any conjugacy class $C$ of $G$ \cite[Theorem 2.5]{MarciniakRitterSehgalWeiss1987}.

For a positive integer $n$ denote by $G[n]$ the set of elements of $G$ of order $n$.
Moreover let $\Cl(G)$ denote the set of all conjugacy classes in $G$ and for $C, D \in \Cl(G)$ and $k$ an integer write $C^k = D$ if $x^k \in D$ for $x \in C$.
The Cohn-Livingstone congruences states that if $u$ is a unit of order $p^m$, with $p$ prime, then for every integer $k$ we have
    \begin{equation}\label{CohnLivingstoneCongruence}
    \pa{u}{G[p^k]} \equiv \begin{cases}
    1 \mod p, & \text{if } k=m; \\ 0 \mod p, & \text{otherwise}.
    \end{cases}
    \end{equation}

This has inspired A. Bovdi to pose the following question.

\begin{quote}
\textbf{Bovdi Problem (BP):} \cite[p. 26]{BovdiBook} (cf. also \cite[Problem~1.5]{ArtamonovBovdi1989})
Let $u$ be a torsion element of $\V(\Z G)$ of prime power order $p^m$.
Is $\pa{u}{G[p^k]}=0$ for every $k\ne m$ and $\pa{u}{G[p^m]}=1$?
\end{quote}

This question was studied for special classes of groups e.g. in \cite{Juriaans1995} or \cite[Proposition 6.5]{HertweckBrauer}. Clearly as an intermediate step between the Zassenhaus Conjecture and the Spectrum Problem one could pose the following generalization of the question of Bovdi.

\begin{quote}
\textbf{General Bovdi Problem (Gen-BP):}
Let $u$ be a torsion element of $\V(\Z G)$ of order $n$.
Is $\pa{u}{G[m]} = 0$ for every $m\ne n$?
\end{quote}

This is also stated as Problem 44 in \cite{Sehgal1993}. It was studied for certain infinite groups in \cite{Dokuchaev1992} and proved for metabelian non-necessarily finite groups in \cite[Corollary 1.4]{DokuchaevSehgal1994}.
Let $u \in \V(\Z G)$ be a torsion unit of order $n$ and let $G$ be finite.
More evidence on (Gen-BP) is that it holds for $m=1$, by the Berman-Higman Theorem, and if $m$ does not divide $n$, by 
a theorem of Hertweck which states that $\pa{u}{C} = 0$ if the order of the elements in $C$ does not divide $n$ 
\cite[Theorem~2.3]{HertweckBrauer}.
We will use these facts in the following without further mention.

A folklore congruence for $a\in \Z G$, a prime $p$, an integer $k$ and $C \in \Cl(G)$ is the following:
    \begin{equation}\label{FolkoreCongruence}
    \sum\limits_{\substack{D \in \Cl(G), \\ D^{p^k} = C}} \pa{a}{D} \equiv \pa{a^{p^k}}{C} \mod p.
    \end{equation}
(See \cite[Remark 6]{BovdiHertweck2008} for some references and \cite[Proposition 3.1]{HeLPPaper} for a proof for the case where $a\in \V(\Z G)$ which actually works in general.)
This suggests to introduce the following property, in the spirit of (BP) and (Gen-BP).

\begin{quote}
\textbf{The Partial Augmentation Power Property (PAP)}:
We say that an element $u$ of $\V(\Z G)$ satisfies the \emph{Partial Augmentation Power Property} (PAP Property for short) if the following holds for every integer $k$:
$$
\PAP(k) : \hspace{.5cm} \pa{u^k}{C} = \sum_{\substack{D\in \Cl(G), \\ D^k=C}} \pa{u}{D} \text{ for every } C\in \Cl(G).
$$
We say that $G$ satisfies the \emph{PAP Property} if any torsion element of $\V(\Z G)$ satisfies the PAP Property.
\end{quote}

Note that each element of $G$ satisfies the PAP Property. Moreover, the partial augmentations of rationally conjugate units coincide. Thus the PAP Property holds for every unit rationally conjugate to an element of $G$ and, in particular, the Zassenhaus Conjecture implies the PAP Property.

In Section~\ref{SectionConnections} we summarize the logical connections between the problems mentioned above.
In Section~\ref{SectionPAP} we prove the PAP Property for a particularly interesting class of units and give
some practical features of the PAP Property.
In Section~\ref{SectionPAPHeLP} we analyze the implications and connections between the PAP Property and the HeLP Method.
We show how these can be applied to the Zassenhaus Conjecture by proving the following theorem.

\begin{theorem}\label{TheoremZC}
Let $G$ be a finite group with an abelian subgroup $A$ of prime index in $G$. Assume that $A$ is generated by at most two elements.
Then the Zassenhaus Conjecture holds for $G$.
\end{theorem}

\section{Connections between the problems}\label{SectionConnections}

The logical connections between the problems are collected in the following proposition.

\begin{proposition}\label{AllProblems} For a finite group $G$ the implications pictured in Figure~\ref{Logic} hold.
\begin{figure}[h]
\centering
\begin{equation*}
  \xymatrix{
  (ZC)  \ar@{=>}[r] & (PAP) \ar@{=>}[r] & (KP) \ar@{<=>}[r] & (Gen-BP) \ar@{=>}[r] \ar@{=>}[d]  & (SP) \ar@{=>}[r] & (PQ)  \\
  &  & & (BP) & &
}
\end{equation*}
\caption{Implications between some Zassenhaus Conjecture related questions.}
\label{Logic}
\end{figure}
\end{proposition}

\begin{proof}
Only the fact that (PAP) implies (Gen-BP) and that (KP) and (Gen-BP) are equivalent needs additional explanation.

Suppose that $G$ satisfies the PAP Property.
Let $u$ be an element of order $n$ in $\V(\Z G)$.
If $d$ is a proper divisor of $n$ then $u^d\ne 1$ and hence, by our assumption we have
\[0=\pa{u^d}{1} = \sum_{\substack{C\in \Cl(G), \\ C^d=1}} \pa{u}{C}.\]
Arguing by induction on $d$ we deduce that for every proper divisor $d$ of $n$ we have
\[\pa{u}{G[d]}=\sum_{\substack{C \in \Cl(G), \\ C \subseteq G[d]}} \pa{u}{C}=0.\]
On the other hand if $C$  is a conjugacy class of $G$ formed by elements of order not dividing $n$ then by the above mentioned \cite[Theorem 2.3]{HertweckBrauer} $\pa{u}{C}=0$.
We conclude that if $m\ne n$ then $\pa{u}{G[m]}=0$.

Assume that (KP) has a positive solution for $G$ and let $u$ be an element of order $n$ in $\V(\Z G)$.
By assumption, $G$ is contained as a subgroup in a finite group $H$ such that $u$ is conjugate to an element of $G$ in $\Q H$.
Since the support of $u$ clearly consists only of elements in $G$ and $G[m]=G\cap H[m]$, we have $\varepsilon_{G[m]}(u) = \varepsilon_{H[m]}(u)=0$ for any integer $m\ne n$. So the (Gen-BP) property follows for $u$.

Suppose that (Gen-BP) holds for $G$, i.e. every torsion element of $\V(\Z G)$ satisfies the property given in (Gen-BP). Let $u$ be a torsion unit in $\V(\Z G)$.
Embed $G$ in the natural way in the symmetric group acting on $G$, i.e. using the regular action of $G$ on itself. Call this symmetric group $H$. Then the cycle type of an element of order $k$ in $G$ viewed as an element of $H$ is given as $|G|/k$ cycles of length $k$. So two elements of $G$ of the same order are conjugate in $H$. In particular all conjugacy classes of elements of the same order in $G$ fuse into one conjugacy class in $H$. Since $u$ and all its powers satisfy the property stated in (Gen-BP) this means that for any power $u^k$ of $u$ there exists exactly one conjugacy class $C$, consisting of elements of the same order as $u^k$, in $H$ such that $\varepsilon_C(u^k) = 1$ and $\varepsilon_{C'}(u^k) = 0$ for any other conjugacy class $C'$.
Moreover, $C$ contains an element from $G$. So $u$ is conjugate in $\Q H$ to an element of $G$ by the above cited \cite[Theorem 2.5]{MarciniakRitterSehgalWeiss1987}.
\end{proof}

One fact connected to (Gen-BP) was recently realized by Kimmerle and Konovalov which follows from a proof of Hertweck \cite{HertweckOrders}.

\begin{proposition} \cite{KimmerleKonovalov2016}
Let $G$ be a finite group and assume that for every element $u \in \V(\Z G)$ of finite order $n$ there exists $C \in \Cl(G)$ consisting of elements of order $n$ such that $\pa{u}{C} \neq 0$. Then this property also holds for solvable extensions of $G$.

In particular, if (Gen-BP) holds for $G$ then (SP) holds for solvable extensions of $G$.
\end{proposition}

\section{The PAP Property}\label{SectionPAP}

In this section we first analyze some features of the PAP Property and then we show that the PAP Property holds for a particularly interesting class of units.
First observe that the rational conjugacy criterion of Marciniak-Ritter-Sehgal-Weiss can be relaxed in the presence of the PAP Property.

\begin{proposition}\label{ZassenhausVersusPAP}
Let $G$ be a finite group. If a torsion element $u \in \V(\Z G)$ satisfies the PAP Property then $u$ is rationally conjugate to an element of $G$ if and only if $\pa{u}{C}\ge 0$ for every $C \in \Cl(G)$.
\end{proposition}

The next lemma collects a few elementary properties of the PAP Property, the most practical of which is that to verify the PAP Property
we only have to check condition $\PAP(d)$ for $d$ dividing the order of $u$.

\begin{lemma}\label{PAPSimplified}
Let $G$ be a finite group and $u$ a torsion element of order $n$ in $\V(\Z G)$.
	\begin{enumerate}
	\item \label{PAPCoprime} If $k$ is an integer coprime with $n$ then $u$ satisfies $\PAP(k)$.
	\item \label{PAP1n} $u$ satisfies $\PAP(1)$ and $\PAP(n)$.
	\item \label{PAPDivisors} If $u$ satisfies $\PAP(d)$ for every divisor d of $n$ then  $u$ satisfies the PAP Property.
	\item\label{PAPPrimeOrder} If $n$ is prime then $u$ satisfies the PAP Property.
	\end{enumerate}
\end{lemma}

\begin{proof}
\eqref{PAPCoprime}
Let $C\in \Cl(G)$ and let $m$ denote the order of the elements in $C$.
If $m\nmid n$ then the order of the elements of every $D\in \Cl(G)$ with $D^k=C$ does not divide $n$ and hence 
$\pa{u^k}{C}=\pa{u}{D}=0$ for every such $D$ by the above mentioned \cite[Theorem 2.3]{HertweckBrauer}. Then the formula 
defining $\PAP(k)$ holds for $C$.

Suppose otherwise that $m\mid n$ and let $l$ be an integer such that $kl\equiv 1 \mod n$. Let $D \in \Cl(G)$ be such 
that $D^k = C$ and $\pa{u}{D} \ne 0$. 
Then the order of the elements in $D$ divides $n$, again by \cite[Theorem 2.3]{HertweckBrauer}, and so the order of 
elements in $D$ and $D^k$ is the same. Hence there is exactly one class which satisfies the properties of $D$, namely 
$C^l$, and we have to prove that $\pa{u^k}{C}=\pa{u}{C^l}$.
Define the $\Z(\GEN{u}\times G)$-module $M$ with $\Z G$ as underlying additive group and product given as follows:
\[(v,g)x=vxg^{-1}, \quad v\in \GEN{u}, \ g\in G, \ x\in M=\Z G.\]
As $M$ is free as $\Z$-module there is a $\Z$-representation $\rho$ associated to this module structure. The character $\chi$ afforded by $\rho$ is given by the following formula:
\begin{equation}\label{CharacterPA}
\chi(v,g)=|C_G(g)| \; \pa{v}{g^G} \quad (v\in \GEN{u}, \ g \in G)
\end{equation}
(see \cite[38.12]{Sehgal1993} or \cite[Lemma~1]{Weiss1991}).
Let $g\in C$. By assumption the order of $(u^k,g)$ is $n$ and $C_G(g)=C_G(g^l)$.
Then the eigenvalues of $\rho(g)$ belong to $\GEN{\zeta_n}$, where $\zeta_n$ denotes a primitive $n$-th root of unity.
As $\gcd(l,n)=1$, there is an automorphism $\sigma$ of $\Q(\zeta_n)$ given by $\sigma(\zeta_n)=\zeta_n^l$. Then $\chi(u^k,g)=\sigma(\chi(u,g^l))$. Thus, applying \eqref{CharacterPA} we have  $\pa{u^k}{C}=\pa{u^k}{g^G} = \pa{u}{(g^l)^G}=\pa{u}{C^l}$, as desired.

\eqref{PAP1n}
$\PAP(1)$ is clear. If $D\in \Cl(G)$, with
$D^n\ne 1$ then $\pa{u}{D}=0$. Therefore, if $C\in \Cl(G)\setminus \{1\}$ then
\[\sum_{\substack{D\in \Cl(G),\\ D^n=C}} \pa{u}{D} = 0 = \pa{1}{C} = \pa{u^n}{C}.\]
Moreover,
\[\sum_{\substack{D\in \Cl(G),\\ D^n=1}} \pa{u}{D}=\sum_{D\in \Cl(G)} \pa{u}{D} = 1 = \pa{1}{1}=\pa{u^n}{1}.\]
Therefore $\PAP(n)$ holds.

\eqref{PAPDivisors} Suppose that $\PAP(d)$ holds for every $d\mid n$ and let $k$ be an arbitrary integer.
Let $d=\gcd(n,k)$. Then $\gcd\left(\frac{n}{d},\frac{k}{d}\right)=1$ and the orders of $u^k$ and $u^d$ are both $\frac{n}{d}$.
Let $C\in \Cl(G)$ and let $m$ be the order of the elements of $C$.

Assume that $m\nmid \frac{n}{d}$. Then $\pa{u^k}{C}=0$.
Furthermore, if $D\in \Cl(G)$ with $D^k=C$ then the order of the elements in $D$ does not divide $n$.
Thus $\pa{u}{D}=0$ and hence $\PAP(k)$ holds in this case.

Suppose otherwise that $m$ divides $\frac{n}{d}$. 
Let $h$ be an integer such that $\frac{k}{d}h \equiv 1 \mod \frac{n}{d}$. 
Let $D \in \Cl(G)$ such that $\pa{u}{D} \ne 0$. Then the order of the elements in $D$ divides $n$ and so
$D^k=C$ if and only if $D^d=C^h$.
Using \eqref{PAPCoprime} and the hypothesis that $\PAP(d)$ holds we deduce that
	$$
	\pa{u^k}{C} = \pa{(u^d)^{k/d}}{C} = \pa{u^d}{C^h} = \sum_{\substack{D\in \Cl(G),\\ D^d = C^h}} \pa{u}{D}
			= \sum_{\substack{D\in \Cl(G), \\ D^k = C}} \pa{u}{D}.
	$$
Here the second equality follows as in the poof of \eqref{PAPCoprime}, noting that the order of $u^d$ is $\frac{n}{d}$.

\eqref{PAPPrimeOrder} is a consequence of \eqref{PAP1n} and \eqref{PAPDivisors}.
 \end{proof}

If $N$ is a normal subgroup of $G$ let $\V(\Z G,N)$ denote the group of units of $\Z G$ mapping to the identity by the linear span of the natural homomorphisms $G\rightarrow G/N$.
We now prove the PAP Property for the torsion elements of $\V(\Z G,N)$ for $N$ a nilpotent normal subgroup.

\begin{theorem}\label{PAPower}
Let $G$ be a finite group and $N$ a normal nilpotent subgroup of $G$. Then every torsion element $u$ of $\V(\Z G,N)$ satisfies the PAP Property.
\end{theorem}

\begin{proof}
For every $C \in \Cl(G)$ fix some $x_C \in C$ and define a subgroup $[C]$ of $\langle u \rangle \times G$ as $[C] = \langle (u,x_C) \rangle$.
Let $\chi$ be the character of $\GEN{u}\times G$ defined in the proof of Lemma~\ref{PAPSimplified}.
This character has been calculated in \cite[Lemma~3.4]{MargolisdelRioCW1}
under the assumption of the theorem.
More precisely we have
\begin{equation}\label{CharacterUnitInZGN}
\chi=\sum_{C \in \Cl(G)} \pa{u}{C} \; \ind_{[C]}^{\GEN{u}\times G}(1).
\end{equation}
Fix some integer $k$, some conjugacy class $C$ and let $g \in C$. Note that for any $D \in \Cl(G)$ we have
    \[\ind_{[D]}^{\GEN{u}\times G}(1)(u^k,g) = \begin{cases}1, & \text{if } g \in  D^k; \\ 0, & \text{otherwise};\end{cases}\]
and $g \in D^k$ if and only if $C = D^k$. So using \eqref{CharacterPA} and \eqref{CharacterUnitInZGN} we obtain
\begin{align*}
|C_G(g)| \pa{u^k}{C} &= \chi(u^k, g) = \sum_{D \in \Cl(G)} \pa{u}{D} \ \ind_{[D]}^{\GEN{u}\times G}(1)(u^k,g) \\
 &= |C_G(g)| \sum_{\substack{D\in \Cl(G), \\ D^k = C}}  \pa{u}{D}.
\end{align*}

Thus $u$ satisfies the PAP Property.
\end{proof}

Combining Theorem~\ref{PAPower} and \cite[Lemma~5.9]{MargolisdelRioCW1} we deduce the following:

\begin{corollary}\label{PAPNilpotentIndexp}
Let $G$ be a finite group with a normal nilpotent subgroup of prime index. Then $G$ satisfies the PAP Property.
\end{corollary}

The units satisfying the hypothesis of Theorem~\ref{PAPower} are the focus of attention in the following question.

\begin{quote}
\textbf{Sehgal's Problem} \cite[Research Problem 35]{Sehgal1993}:
Let $N$ be a normal nilpotent subgroup of the finite group $G$. Is every torsion element $u$ of $\V(\Z G,N)$ rationally conjugate to an element in $G$?
\end{quote}

This problem is motivated by a typical strategy when studying the Zassenhaus Conjecture
for groups possessing a ``good'' normal subgroup $N$ which consists in first considering the units in $\V(\Z G,N)$.
It is known that Sehgal's Problem has a positive solution if $N=G$ (and hence $G$ is nilpotent) \cite{Weiss1991},
if $N$ a $p$-group \cite[Proposition~4.2]{Hertweck2006} and under some other particular conditions
(see Theorem~2.4 in \cite{MargolisdelRioCW1} for references and
\cite{MargolisdelRioCW1} and \cite{MargolisdelRioCWAlgorithm} for more results).

\begin{remark}
We have stated the PAP Property only for torsion units but it makes sense for arbitrary elements of a group ring. However one can easily give units of infinite order not satisfying the statement of the PAP Property.
A counterexample to $\PAP(2)$ is $u=1-g+g^2$ with $G=\GEN{g}$ of order coprime with $6$.
This is a unit because it is the result of evaluating the 6-th cyclotomic polynomial at $g$
(see \cite[Proposition~3.1]{MarciniakSehgal2005} or \cite[Proposition~8.1.3]{GRG1}).
\end{remark}

\section{The PAP adapted HeLP Method}\label{SectionPAPHeLP}

In this section we will explain how to adapt the HeLP Method for units satisfying the PAP Property.
First of all we explain the HeLP Method which was introduced by Luthar and Passi in \cite{LutharPassi1989} for ordinary characters, and extended by Hertweck in \cite{Hertweck2008} to Brauer characters.

Fix a finite group $G$ and a positive integer $n$.
Let $\mu_n$ denote the set of complex $n$-th roots of unity.
Let $\Irr(G)$ be the set of ordinary irreducible characters of $G$.
If $p$ is a prime integer then let $G_{p'}$ denote the set of $p$-regular elements of $G$ and let $\IBr_p(G)$ denote the set  of irreducible $p$-Brauer characters of $G$.
We extend the Brauer characters to $G$ by setting $\chi(g)=0$ for $\chi$ a $p$-Brauer character of $G$ and $g\in G\setminus G_{p'}$.
Let
    $$\IRR_n(G) = \Irr(G) \cup \bigcup_{\substack{p \mid |G|, \\ p \nmid n}} \IBr_p(G).$$
Let $\Z_n(G)$ denote the set of arrays $X=(X_{d,C})_{d\mid n,C\in \Cl(G)}$ with integral entries such that
    $$\sum_{C\in \Cl(G)} X_{d,C}=1 \text{ for every } d\mid n.$$
For every $\chi\in \IRR_n(G)$, every $X=(X_{d,C})_{d\mid n,C\in \Cl(G)}\in \Z_n(G)$ and every $\xi\in \mu_n$ let us denote
    \begin{equation}\label{multiplicity}
      \mu(\chi,X,\xi) = \frac{1}{n}\sum_{d\mid n} \sum_{C \in \Cl(G)}  \text{Tr}_{\mathbb{Q}(\zeta^d)/\mathbb{Q}}\left(\chi(C)\xi^{-d}\right) X_{d,C}.
    \end{equation}
Let
    $$\VPA_n(G) = \{X \in \Z_n(G) : \mu(\chi,X,\xi) \in \Z^{\ge 0}, \text{ for every } \chi \in \IRR_n(G), \xi\in \mu_n\}.$$

Recall that in the constructions of the $p$-Brauer characters one fixes an isomorphism $\xi\mapsto \overline{\xi}$ between $\mu_m$, with $m$ the maximum divisor of the exponent of $G$ which is coprime with $p$, and the group of $m$-th roots of unity in a big enough field of characteristic $p$.

Let $u$ be an element of order $n$ in $\V(\Z G)$ and set $X(u)=(\pa{u^d}{C})_{d\mid n,C\in \Cl(G)}$.
Clearly $X(u)\in \Z_n(G)$ since $u$ and any of its powers is a unit of augmentation $1$.
We denote
  $$\mu(\chi,u,\xi) = \mu(\chi,X(u),\xi).$$
By results of Luthar and Passi \cite{LutharPassi1989} and Hertweck \cite{HertweckBrauer}, if $\chi\in \Irr(G)$ (respectively, $\chi\in \IBr_p(G)$) then $\mu(\chi,u,\xi)$ is precisely the multiplicity of $\xi$ (respectively, of $\overline{\xi}$) as an eigenvalue of $\rho(u)$, where $\rho$ is an ordinary representation (respectively, a $p$-modular representation) of $G$ affording the character (respectively, the Brauer character) $\chi$. Thus $X(u)\in \VPA_n(G)$.
Moreover, if $C$ is a conjugacy class of $G$ and $g\in C$  we denote
\begin{equation}\label{MultiplicityInGroup}
\mu(\chi,C,\xi)=\mu(\chi, g, \xi) = \frac{1}{n}\sum_{d | n}  \text{Tr}_{\mathbb{Q}(\zeta^d)/\mathbb{Q}}\left(\chi(g^d)\xi^{-d}\right).
\end{equation}

For a general $X\in \VPA_n(G)$, one envisions its entry $X_{d,C}$ as the partial augmentation of $u^d$ at $C$ for a ``potential'' element $u$ of order $n$ in $\V(\Z G)$.
This is enforced by considering $\VPA$ as an acronym for ``virtual partial augmentations''.
By the Marciniak-Ritter-Sehgal-Weiss Criterion, if the entries of each element of order $n$ in $\VPA_n(G)$ are all non-negative, then every unit of $\V(\Z G)$ is rationally conjugate to an element of $G$.
In that case we say that \emph{the HeLP Method proves the Zassenhaus Conjecture} for units of order $n$ in $\V(\Z G)$.

In fact, if one realizes the above procedure only with the characters $\Irr(G)$, one obtains exactly the partial augmentations of torsion units of augmentation 1 in the complex group algebra $\C G$ which have integral partial augmentations.

Actually, the HeLP Method is not applied exactly as explained above.
In practice one uses all the available information for the units of order $n$ in $\V(\Z G)$.
For example, by the Berman-Higman Theorem, and the theorem of Hertweck mentioned above one usually reduces $\VPA_n(G)$ by imposing the following for every $X\in \VPA_n(G)$:
    $$X_{d,C}= 0, \text{ if } C=g^G \text{ and } \begin{cases}
    \text{either } d\ne n \text{ and } 1\in C, \text{ or } \\
    \frac{n}{d} \text{ is not a multiple of } |g|.
    \end{cases}$$
Furthermore, $\VPA_n(G)$ might be further refined with all the available information in the particular case considered.
For example, by the Cohn-Livingstone Theorem to prove the Zassenhaus Conjecture using HeLP it is enough to consider the divisors $n$ of the exponent of $G$. If additionally, we know that the Spectrum Problem has a positive solution for $G$ then it is enough to consider only the elements whose order is in the spectrum of $G$.
Other refinements might be obtained from the congruences given in \eqref{CohnLivingstoneCongruence} and \eqref{FolkoreCongruence}.

So there is a plain HeLP Method, as explained above, and several adapted HeLP Methods, depending on the information available.

The (adapted) HeLP Method can be used not only to prove the Zassenhaus Conjecture but also to treat other questions about torsion elements of $\V(\Z G)$. For example, one can use it to deal with the Spectrum Problem by proving that if $\VPA_n(G)\ne \emptyset$ then $G$ has an element of order $n$. Similarly, using Lemma~\ref{PAPSimplified}, one can check the PAP Property for $G$ by checking that for every divisor $n$ of the exponent of $G$ each element $(X_{d,C})$ of $\VPA_n(G)$ satisfies the following equality:
    \begin{equation}\label{PAPVirtual}
    X_{d,C} = \sum_{\substack{D\in \Cl(G), \\ D^d=C}} X_{1,D}.
    \end{equation}
In this case we say that the HeLP Method proves the PAP Property.

There are well known examples in the literature where the HeLP Method does not provide positive answers to the Prime Graph Question, and hence for these examples the HeLP Method does not prove the PAP Property (cf. e.g. \cite{4primaryII} or \cite{BovdiKonovalovConway}).
We now present some examples from the literature for which the HeLP Method gives positive answers to the General Bovdi Problem but does not prove the PAP Property.

\begin{example}{\rm
Let $p$ and $t$ be different prime integers with $p\ne 2$ and $t\ge 5$ and let $f$ be a positive integer with $p^f \equiv \pm 1 \mod 4t$. Let $G=\PSL(2,p^f)$. Then $G$ has a conjugacy class $C$ of order $2t$ and, by the main result of \cite{delRioSerrano}, there is $X\in \VPA_{2t}(G)$ such that the only non-zero entries of $X$ are the following:
    $$X_{2t,1}=X_{t,C^t}=X_{2,C^2}=X_{1,C^{\frac{t-1}{2}}}=X_{1,C^{\frac{t+1}{2}}}=1, \quad X_{1,C^{t-1}}=-1.$$
Let $D\in \Cl(G)$, consisting of elements of order $2t$. Then $D^2=C^2$ if and only if $D=C$ or $D=C^{t-1}$.
Therefore
\[\sum_{\substack{D\in \Cl(G), \\ D^2=C^2}} X_{1,D} = X_{1,C}+X_{1,C^{t-1}}=-1\ne 1=X_{2,C^2}.\]
So, the HeLP Method does not prove the PAP Property for units of order $2t$ in $\V(\Z G)$.

However, as every element in $\VPA_{2t}(G)$, which is not of the form $X(g)$ for some $g\in G$, is of the form given above \cite{delRioSerrano}, we deduce that the HeLP Method gives a positive answer to the General Bovdi Problem in $G$ for units of order $2t$.
} \qed \end{example}

\begin{example}{\rm The torsion units in integral group rings of groups of order at most 287 have been studied in \cite{SmallGroups} and for each of these groups $G$ and all the possible values of $n$, the elements of $\VPA_n(G)$, which are not of the form $X(g)$ for $g \in G$, have been listed there, using the HeLP Method in several adapted versions. It turns out that the adapted HeLP Method proves the PAP Property for all these groups except for the group $G$ whose catalogue number in the Small Group Library of \texttt{GAP} \cite{GAP} is (216, 153).
We use the \texttt{GAP} notation for conjugacy classes to elaborate.
When applying the HeLP Method here one finds $X\in \VPA_6(G)$ with $X_{1,\texttt{3a}}=X_{1,\texttt{6a}}=1$, $X_{1,\texttt{3d}}=-1$, $X_{1,C}=0$ for each $C\in \Cl(G)\setminus \{\texttt{3a},\texttt{3d},\texttt{6a}\}$ and $X_{2,C}\in \{0,1\}$ for all $C\in \Cl(G)$.
Moreover, $(\texttt{3a})^2=(\texttt{6a})^2=\texttt{3c}$ and $(\texttt{3d})^2=\texttt{3e}$.
Thus
\[\sum_{\substack{C \in \Cl(G), \\ C^2 = \texttt{3c}}} \pa{u}{C} = \pa{u}{\texttt{3a}} + \pa{u}{\texttt{6a}} = 2 \ne X_{2,\texttt{3c}}\]
and
\[\sum_{\substack{C \in \Cl(G), \\ C^2 = \texttt{3e}}} \pa{u}{C} = \pa{u}{\texttt{3d}} = -1 \ne X_{2,\texttt{3e}}.\]
Therefore, the HeLP Method does not prove the PAP Property for this group.

Note that also this example satisfies the statement of the General Bovdi Problem.
\qed }\end{example}

Now we present a version of the HeLP Method adapted to the situation where the unit satisfies the PAP Property, as for example under the hypothesis of Theorem~\ref{PAPower}.
Recall that we view an element $X=(X_{d,C})$ of $\VPA_n(G)$ as the list of partial augmentations of the $u^d$ with $d\mid n$, for a hypothetical element $u$ of order $n$ in $\V(\Z G)$.
In case $G$ satisfies the PAP Property, or at least the elements of order $n$ in $\V(\Z G)$ do, we may assume that $X$ satisfies \eqref{PAPVirtual} and hence
$X$ is completely determined by $(X_{1,C})_{C\in \Cl(G)}$.
So, for the PAP-adapted HeLP Method we replace $\Z_n(G)$ with the set $\Z_n^{\PAP}(G)$ formed by the lists $(Y_C)_{C\in \Cl(G)}$ of integers
satisfying:
\begin{itemize}
 \item $\sum_{C\in \Cl(G)} Y_C = 1$, and
 \item $Y_C=0$, if either $1\in C$ or $n$ is not multiple of $|g|$ for $g\in C$.
\end{itemize}
Given $Y=(Y_C)\in \Z_n^{\PAP}(G)$, there is a unique $X(Y)=(X_{d,C})\in \Z_n(G)$ such that $Y=(X_{1,X})_{C\in \Cl(G)}$ and  condition \eqref{PAPVirtual} holds for every $d\mid n$, namely:
	$$X_{d,C} = \sum_{\substack{D\in \Cl(G), \\ D^d=C}} Y_D.$$
In that case we have,
	\begin{equation}\label{muPAP}	
	 \mu(\chi, X(Y), \xi) = \sum_{C \in \Cl(G)} Y_C\; \mu(\chi, C, \xi).
	\end{equation}
Indeed,
	\begin{align*}
	\mu(\chi, X(Y), \xi) &= \frac{1}{n}\sum_{d | n} \sum_{C\in \Cl(G)}  \text{Tr}_{\mathbb{Q}(\zeta^d)/\mathbb{Q}}\left(\chi(C)\xi^{-d}\right) X_{d,C}\\
	&= \frac{1}{n}\sum_{d | n} \sum_{C \in \Cl(G)}  \text{Tr}_{\mathbb{Q}(\zeta^d)/\mathbb{Q}}\left(\chi(C)\xi^{-d}\right)
	\sum_{\substack{D \in \Cl(G), \\ D^d = C}} Y_D \\
	&= \frac{1}{n}\sum_{d | n} \sum_{C \in \Cl(G)} \sum_{\substack{D \in \Cl(G), \\ D^d = C}}
\text{Tr}_{\mathbb{Q}(\zeta^d)/\mathbb{Q}}\left(\chi(D^d)\xi^{-d}\right)
	 Y_C \\
	&= \sum_{C \in \Cl(G)} Y_C \frac{1}{n}
	\sum_{d | n}  \text{Tr}_{\mathbb{Q}(\zeta^d)/\mathbb{Q}}\left(\chi(C^d)\xi^{-d}\right) \\
	&= \sum_{C \in \Cl(G)} Y_C \; \mu(\chi, C, \xi).
	\end{align*}

This proves the following:

\begin{proposition}\label{MultiplicitiesLinear}
If $u$ is an element of order $n$ in $\V(\Z G)$ satisfying the PAP Property, $\chi\in\IRR_n(G)$ and $\xi\in \mu_n$ then
	\[\mu(\chi, u, \xi) = \sum_{C \in \Cl(G)} \pa{u}{C}\mu(\chi, C, \xi).\]
\end{proposition}

This justifies the correctness of the \emph{PAP-adapted HeLP Method} which consists in calculating
the following set
    \begin{equation}\label{VPAPAP}
    \VPA_n^{\PAP}(G)=\left\{Y\in \Z_n^{\PAP}(G): \matriz{{l} \sum_{C \in \Cl(G)} Y_C\; \mu(\chi, C, \xi) \ge 0, \\\text{ for all } \chi \in \IRR_n(G), \xi \in \mu_n} \right\}
    \end{equation}
for every divisor $n$ of the exponent of $G$.
This will provide the partial augmentations of all the ``potential'' elements of order $n$ in $\V(\Z G)$ satisfying the PAP Property.
In case all the elements of $\VPA_n^{\PAP}(G)$ have all non-negative entries then all the elements of order $n$ in $\V(\Z G)$ satisfying the PAP Property are
rationally conjugate to elements of $G$.

The PAP-adapted HeLP Method has the following advantages, with respect to the standard HeLP Method:
The number of unknowns is smaller than for the standard HeLP Method, unless $n$ is prime.
The condition $\mu(\chi,X,\xi)\in \Z^{\ge 0}$ from the standard HeLP Method is replaced by $\sum_{C \in \Cl(G)} Y_C\; \mu(\chi, C, \xi) \ge 0$ which seems to
be more suitable for implementation.

We illustrate the PAP-adapted HeLP Method in the following proof of Theorem~\ref{TheoremZC}.
\medskip

\begin{proofof}\textit{Theorem~\ref{TheoremZC}}.
First of all $G$ satisfies the PAP Property by Corollary~\ref{PAPNilpotentIndexp}.
Let $[G:A] = p$ and let $x \in G$ be an element of $p$-power order such that $G = \langle A, x \rangle$. For a prime $q$ denote by $A_q$ a Sylow $q$-subgroup of $A$. We will also denote the $q$-part of an element $g \in G$ by $g_q$. 

Let $u$ be a torsion element of $\V(\Z G)$. 
If $u\not\in \V(\Z G,A)$ then $u$ is rationally conjugate to an element in $G$ by \cite[Lemma~5.9]{MargolisdelRioCW1}. 
So assume that $u\in \V(\Z G,A)$.
Let $a\in G$ with $\pa{u}{a^G}\ne 0$.
By results of Hertweck \cite[Theorem 2.3]{MargolisdelRioCW1} we have $a \in A$ and whenever $g$ is an element of $G$ 
satisfying $\pa{u}{g^G} \ne 0$ the $q$-parts $a_q$ and $g_q$ are conjugate for any prime $q$.

For every prime divisor $q$ of $|A|$ let $b_q, c_q \in A_q$ such that $A_q = \GEN{b_q}\times \GEN{c_q}$ and $a_q \in \langle b_q \rangle$.
Let $K=\prod_q K_q$ with $q$ running on the prime divisors of $|A|$ and
    $$K_q =\begin{cases}
    \GEN{c_q}, & \text{if } x \in N_G(\langle a_q \rangle); \\ \GEN{b_q}, & \text{otherwise}.
    \end{cases}$$
As $A/K$ is cyclic, there is a linear representation $\rho$ of $A$ with kernel $K$. Let $\xi=\rho(a)$.

Note that if $x \notin N_G(\langle a_q \rangle)$ then $\langle a_q \rangle \neq \langle a_q \rangle^{x^i}$ for any $i = 1,...,p-1$.
We claim that if $d$ is an element in $A$ satisfying that $a_q$ and $d_q$ are conjugate in $G$ for every prime $q$ then $\rho(d) = \xi$ implies $d = a$. Indeed, if $\rho(d)= \xi$ then $ad^{-1}\in K$.
Therefore, if $x \notin N_G(\langle a_q \rangle)$ then $a_qd_q^{-1} \in \langle b_q \rangle$. Since $a_q \in \langle b_q 
\rangle$ and the order of $a_q$ and $d_q$ coincide it follows that $\langle a_q \rangle = \langle d_q \rangle$. Now, if 
$d_q \ne a_q$ then the fact that $d_q$ and $a_q$ are conjugate in $G$ and the condition $x \notin N_G(\langle a_q 
\rangle)$ imply that $\langle a_q \rangle \ne \langle d_q \rangle$, giving a contradiction. On the other hand, if $x \in 
N_G(\langle a_q \rangle)$ then $a_qd_q^{-1} \in \langle c_q \rangle \cap \langle a_q \rangle = 1$. Thus $a_q = d_q$ for 
every $q$ and so $a = d$, as desired. 

Therefore, if $g$ is an element of $G$ such that $\pa{u}{g^G}\ne 0$ and $\xi$ is an eigenvalue of $\ind_A^G \rho(g)$ then $g\in a^G$.
So, denoting by $\chi$ the character afforded by $\ind_A^G \rho$, by Proposition~\ref{MultiplicitiesLinear}, we have $\mu(\chi, u, \xi) = \pa{u}{a^G}\mu(\chi, a, \xi) \geq 0$ and hence $\pa{u}{a^G}\ge 0$. As $a$ is an arbitrary element of $G$ with $\pa{u}{a^G}\ne 0$ we conclude that $u$ has non-negative partial augmentations and hence it is rationally conjugate to an element of $G$, by Proposition~\ref{ZassenhausVersusPAP}.
\end{proofof}

\begin{remark}
Proposition~\ref{MultiplicitiesLinear} is the theoretical support for the inequalities in \eqref{VPAPAP} and hence for 
the adapted HeLP Method. 
In fact for ordinary characters these inequalities are those in \cite[Proposition~3.6]{MargolisdelRioCW1}.
We will explain this using the notation from these two results, in particular $u$ is an element of order $n$ in 
$\V(\Z G)$ while $\xi$ is an $n$-th root of unity and $\chi$ an ordinary character of $G$. Define a linear character 
$\eta$ of $\langle u \rangle$ via $\eta(u) = \xi^{-1}$ and a character $\psi$ of $\langle u \rangle \times G$ as $\eta 
\otimes \chi$. Let $m$ be an element in $G$. Then \[\mu(\chi, m, \xi) = \frac{1}{|u|}\sum_{i=0}^{|u|-1} \psi(u^i,m^i) = 
a(m, \psi),  \]
where the latter is notation from \cite[Section~3]{MargolisdelRioCW1}.

On the other hand the elementary tensors of $\Irr(\langle u \rangle) \otimes \Irr(G)$ form the ordinary irreducible 
characters of $\langle u \rangle \times G$, i.e. each $a(m, \psi)$ may be viewed as the sum of multiplicities of certain 
eigenvalues of representations of $G$ evaluated at $m$. 
So the inequalities in \cite[Proposition~3.6]{MargolisdelRioCW1} can be translated to the inequalities in 
\eqref{VPAPAP}.
This shows that the HeLP Method and the inequalities of \cite[Proposition~3.6]{MargolisdelRioCW1} provide the same 
information about Sehgal's Problem.
Moreover, by \cite[Remark~3.10]{MargolisdelRioCW1} the inequalities in \cite[Theorem~3.8]{MargolisdelRioCW1} provide at 
least this information.

However \cite[Proposition~3.6]{MargolisdelRioCW1} only states the result for units in $\V(\Z G,N)$ for some nilpotent 
normal subgroup $N$, while the inequalities in \eqref{VPAPAP} are valid for any unit satisfying the PAP Property. Also 
they can be applied with Brauer characters which is not the case for \cite[Proposition~3.6]{MargolisdelRioCW1}.
\end{remark}

\bibliographystyle{amsalpha}
\bibliography{CW}

\noindent
Departamento de Matemáticas, Universidad de Murcia, 30100 Murcia, Spain\newline
email: leo.margolis@um.es, adelrio@um.es

\end{document}